\documentclass[12pt,letterpaper]{amsart}

\usepackage{mathrsfs}
\usepackage{latexsym}
\usepackage{graphics}
\usepackage{amsmath,amsfonts,amsthm,amssymb,amscd}
\input amssym.def
\input amssym.tex
\newtheorem*{maintheorem}{Theorem 1.1}
\newtheorem*{Theorem}{Theorem}
\newtheorem{theorem}{Theorem}[section]

\newtheorem{corollary}[theorem]{Corollary}
\newtheorem{lemma}[theorem]{Lemma}
\newtheorem{proposition}[theorem]{Proposition}

\theoremstyle{remark}
\newtheorem{remark}[theorem]{Remark}

\newcommand{\ud}{\mathrm{d}}

\newcommand{\abs}[1]{\left|#1\right|}

\renewcommand{\Re}{\textup{Re }} % \textup prevents Re from being italicized in a thm

\newcommand{\floor}[1]{\left\lfloor#1\right\rfloor}
\newcommand{\ceil}[1]{\left\lceil#1\right\rceil}

\newcommand{\printlineno}{\hspace*{-.5in}\texttt{SOURCE LINE \#\the\inputlineno}}
\renewcommand{\printlineno}{}
\printlineno

\newcommand{\odd}{\textrm{ odd}}
\renewcommand{\mod}[1]{{\ifmmode\text{\rm\ (mod~$#1$)}\else\discretionary{}{}{\hbox{ }}\rm(mod~$#1$)\fi}}

\newcommand{\sumstar}[1]{\sideset{}{^*}\sum_{#1}}
\title{The distribution of the maximum of character sums}
\author{Jonathan W. Bober}
\address{Department of Mathematics, University of Washington, Seattle, WA, USA}
\email{jwbober@math.washington.edu}
\date{}
\author{Leo Goldmakher}
\address{Department of Mathematics, University of Toronto, Toronto, ON, Canada}
\email{leo.goldmakher@utoronto.ca}
\begin{document}

\begin{abstract}
We obtain explicit bounds on the moments of character sums, refining estimates of
Montgomery and Vaughan. As an application we obtain results on the distribution
of the maximal magnitude of character sums normalized by the square root of the
modulus, finding almost double exponential decay in the tail of this distribution.
%Three decades ago, Montgomery and Vaughan proved that for most Dirichlet characters the associated character sum is bounded by square-root of the modulus. We revisit their result and make it more precise. This allows us to study the distribution of the maximum magnitude of character sums.

%We revisit a result of Montgomery and Vaughan \cite{montgomery-vaughan-mean-values} which says that for most Dirichlet characters $\chi$ modulo $q$ the maximum of the character sum $\sum_{n \le N} \chi(n)$ is $\ll q^{1/2}$. We quantify this result in order to study the distribution of the size of this maximum.
\end{abstract}

\maketitle

\section{Introduction}

Given $\chi\mod{q}$ a Dirichlet character, we consider character sums of the form
\[
S_\chi(t) = \sum_{n \leq t} \chi(n) .
\]
Character sums play a fundamental role in many number theoretic problems, and have been the object of intense
study over the past century. Despite much progress, our understanding of the behavior of $S_\chi(t)$ is still
far from what is expected to be true. For example, a widespread belief (and a corollary of the generalized
Riemann hypothesis) is that for all nonprincipal $\chi\mod{q}$,
\[
S_\chi(t) \ll_\epsilon t^{1/2} q^\epsilon ;
\]
however, this is only known to hold when $t \gg q^{1 - \epsilon}$. Even worse, when $t \leq q^{1/4 - \epsilon}$, no nontrivial estimates are known outside of special cases (see \cite{graham-ringrose,iwaniec}).

The first nontrivial upper bound was obtained independently in 1918 by P\'{o}lya and Vinogradov, who proved that for any nonprincipal $\chi \mod{q}$,
\begin{equation}
\label{Eq:PolyaVinogradovInequality}
M(\chi) := \max_t |S_\chi(t)|  \ll \sqrt{q} \log q .
\end{equation}
This bound is close to being sharp: for any primitive character $\chi \mod{q}$ we can apply partial summation to the Gauss sum
\[
\tau(\chi) := \sum_{n \leq q} \chi(n) e(n/q)
\]
and use the well known relation $|\tau(\chi)| = \sqrt{q}$ to deduce that
\begin{equation}
\label{Eq:LowerBoundOnM}
M(\chi) \gg \sqrt{q} .
\end{equation}
Thus, $\sqrt{q} \ll M(\chi) \ll \sqrt{q} \log q$. Assuming the GRH,
Montgomery and Vaughan \cite{montgomery-vaughan-upper-bound} improved the upper bound to $M(\chi) \ll \sqrt{q}\log \log q$, and it is known that
$M(\chi)$ gets this large for many characters (see \cite{granville-sound-pret, paley}).
Unconditionally, the upper bound has been improved only for some special families of characters
(see \cite{goldmakher-mult-mimic, goldmakher-smooth-character-sums}); however, a theorem of Montgomery and Vaughan
\cite{montgomery-vaughan-mean-values}
asserts that for many characters, $M(\chi) \asymp \sqrt{q}$. More precisely, they proved that for any $k \geq 2$,
\[
    \frac{1}{\phi(q)}\sum_{\chi \ne \chi_0} M(\chi)^{2k} \ll_k q^k,
\]
where the sum runs over all nonprincipal characters $\chi\mod{q}$. This bound implies that for all sufficiently
large $\alpha$, a positive proportion of characters modulo $q$ satisfy $M(\chi) \le \alpha \sqrt q$. Moreover,
this proportion tends to $1$ as $\alpha$ tends to infinity. On the other hand, this proportion can never
equal $1$; this is implied, for example, by results of Granville and Soundararajan
\cite{granville-sound-extreme-zeta} on the distribution of $L(1,\chi)$, and follows from our Theorem
\ref{main-corollary}.

With this in mind, let $F_q(\alpha)$ denote the proportion of nonprincipal characters modulo $q$ for which
$M(\chi)$ is smaller than $\alpha \sqrt{q}$:
\[
    F_q(\alpha) = \frac{1}{\phi(q) - 1}\#\{\chi\mod{q} : M(\chi) \le \alpha \sqrt{q}, \chi \ne \chi_0\}.
\]
Furthermore, define
\[
    F(\alpha) = \liminf_{q \rightarrow \infty} F_q(\alpha).
\]
From the previous discussion, 
$F(\alpha) \to 1$ as $\alpha \to \infty$. In this paper,
we study the rate at which it does so. We do this by computing explicit bounds for the moments
$\sum_{\chi \ne \chi_0} M(\chi)^{2k}$. Our main theorem is
\begin{theorem}\label{main-theorem}
For fixed $k$ and $q \rightarrow \infty$,
\[
\frac{1}{\phi(q)} \sum_{\chi \ne \chi_0} M(\chi)^{2k} \le \big(C(k) + o(1)\big)q^k,
\]
where
\[
    C(k) = \exp\big(4k \log\log k + k \log\log\log k + O(k)\big).
\]
\end{theorem}

If there are not many primitive characters modulo $q$, this sum may be much smaller. However, when
we restrict to moduli without many prime factors, we can get a lower bound of a similar shape. For simplicity,
we state the result for prime modulus.

\begin{theorem}\label{theorem2}
For fixed $k$ and $q \rightarrow \infty$ over the primes,
\[
\frac{1}{\phi(q)} \sum_{\chi \ne \chi_0} M(\chi)^{2k} \ge \big(c(k) + o(1)\big)q^k,
\]
where
\[
    c(k) = \exp\big(2k \log\log k + O(k)\big).
\]
\end{theorem}

As a consequence of our work, we have almost doubly exponential decay on the number of characters whose
maximal sum gets large.
\begin{theorem}\label{main-corollary}
As $\alpha \rightarrow \infty$,
%\[
%    1 - e^{-e^{\sqrt{B\alpha}}}
%    \le F(\alpha) \le
%    1 - e^{-\left[\frac{2 e^\gamma}{\pi\alpha} \exp\left(\frac{\pi \alpha}{e^\gamma} - A - 1\right)\right]\left(1 + O(\alpha^{-1/2})\right)}
    %1 - \frac{1 + o(1)}{\log \alpha} e^{-3\exp(\sqrt \alpha/\log \alpha)}
    %\frac{1}{\log \alpha C'} e^{-3\exp(\sqrt {\alpha C'}/\log (\alpha C'))}\big(1 + o(1)\big)
%\]
\[
    1 - \exp\big(-e^{B\alpha^{1/2 - o(1)}}\big)
    \le F(\alpha) \le
    1 - \exp\left(-\frac{2 e^\gamma}{\pi\alpha} e^{  {}^{\frac{\pi \alpha}{e^\gamma} - A - 1}}\left(1 + O(\alpha^{-1/2})\right)\right)
    %1 - \frac{1 + o(1)}{\log \alpha} e^{-3\exp(\sqrt \alpha/\log \alpha)}
    %\frac{1}{\log \alpha C'} e^{-3\exp(\sqrt {\alpha C'}/\log (\alpha C'))}\big(1 + o(1)\big)
\]
where $A$ is defined in the statement of Theorem \ref{theorem2b} and $B$ is a positive constant.
\end{theorem}
\begin{remark}
Our argument produces the more precise lower bound
\[
    1 - 
    \exp\Bigg(-\exp\bigg(\frac{B\alpha^{1/2}}{(\log \alpha)^{1/4}}\bigg)\ 
        \bigg(\frac{2\log\log\alpha}{\log\alpha}
        + O\Big(\frac{1}{\log \alpha}\Big)\bigg)\Bigg),
\]
for some constant B.
However, it is not clear whether this lower bound is best possible.
\end{remark}

One may ask further questions about the distribution of the maximum of character sums. Is $F(\alpha)$ continuous?
Can we replace the $\liminf$ with a limit? The answer to the second question seems likely to be no, but it might
be that the limit will exist if we restrict to primes.
%\begin{equation}
%\tilde f(\alpha) =? \lim_{p \rightarrow \infty} f_p(\alpha)
%\end{equation}
One can ask the same questions about related distributions
defined over all characters with modulus $q \le N$,
\[
    G_N(\alpha) :=
    \frac{ \#\{\chi \mod{q} : q \le N, M(\chi) \le \alpha \sqrt{q}\}}{\#\{\chi \mod{q} : q \le N\}}.
\]
It may then be interesting to compare the limiting distributions of $F_q(\alpha)$ and $G_N(\alpha)$.
%In this case, one might expect that $\lim_{N \rightarrow \infty} g_N(\alpha)$ exists and is equal to $f(\alpha)$. (Question for self: Is it easy to see that it must be the case that $\liminf g_N = \liminf f_q$?)

\begin{figure}
\includegraphics{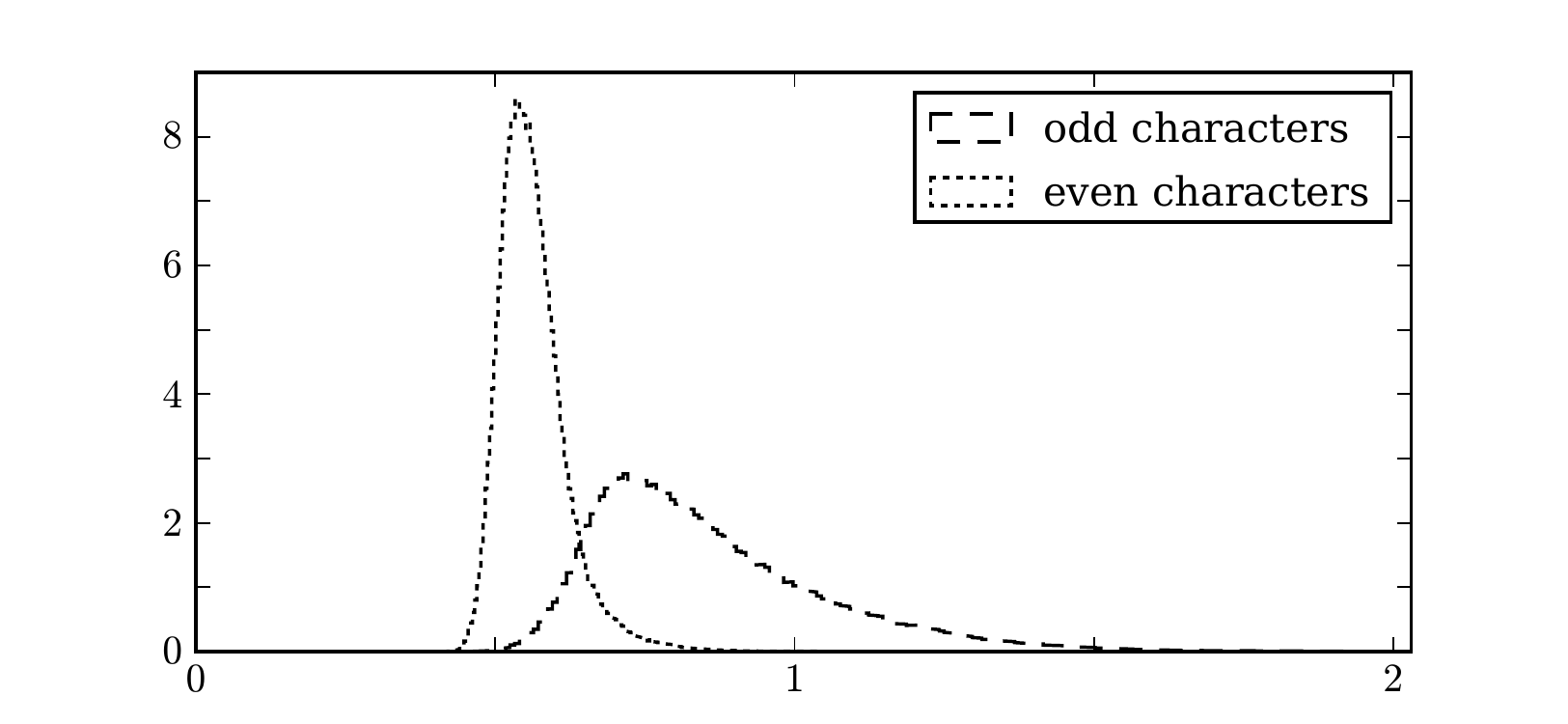}
\caption{Normalized histograms of $M(\chi)/\sqrt{q}$ for $q = 1500007$, split into even and odd characters. Note
that more than 99.999\% of the characters modulo $q$ satisfy $M(\chi)/\sqrt{q} < 2$.}
\end{figure}

\subsection{Acknowledgements}
We would like to thank D. Koukoulopoulos and Y. Lamzouri for helpful discussions, and the referee
for important corrections. Much of this work was completed while the first author was a member of the Institute for Advanced Study and the second author was a visitor there. We are grateful for their hospitality. The first author was supported by the National Science Foundation under grant DMS-0635607 while a member of the Institute for Advanced Study from September 2009 to December 2010, partially supported by NSF grant DMS-0441170, administered by the Mathematical Sciences Research Institute while in residence at MSRI during the Arithmetic Statistics program in Spring 2011, and is currently supported by the NSF under FRG Grant No. DMS-0757627. The second author was partially supported by NSERC grant A5123.

\section{A primer: the distribution of $|S_\chi(q/2)|$}

In this section we derive an estimate of the $2k$-th moment of $|S_\chi(q/2)|$ and use it to derive
results on its distribution. These are similar to results of Granville and Soundararajan \cite{granville-sound-extreme-zeta} on the distribution of $\abs{L(1, \chi)}$, and
we use parts of their work. However, we refrain from giving any uniformity in $q$.
In the
next section, we will develop these ideas to find an upper bound for moments of $M(\chi)$.

\begin{theorem}\label{theorem2b}
For a prime $q$,
\begin{multline}
\frac{1}{\phi(q)} \sum_{\chi \ne \chi_0} \abs{S_\chi(q/2)}^{2k} =
    \left(\frac{e^\gamma}{\pi}\right)^{2k} q^k (\log k)^{2k}
    \exp\left(\frac{2k}{\log k}\left(A + O\left(\frac{1}{\log k}\right)\right) \right) \\
	 \\
    + O_{k, \epsilon}\big(q^{k - 1 + \epsilon}\big).
\end{multline}
Here the sum runs over all nonprincipal characters $\chi\mod{q}$ and
\[
    A = \int_0^2 \frac{\log I_0(t)}{t^2} \ud t,
\]
where
\begin{equation}
\label{eq:Bessel}
    I_0(t) = \sum_{n=0}^\infty \frac{t^{2n}}{2^{2n}(n!)^2}
\end{equation}
is a modified Bessel function.
\end{theorem}
\begin{remark}
Theorem \ref{theorem2} follows immediately from the lower bound of Theorem \ref{theorem2b}.
\end{remark}
\begin{remark}
We restrict to prime modulus in this theorem in order to avoid dealing with imprimitive characters. When the modulus is not prime we can get the same quality upper bound, though perhaps with a slightly larger constant in place of $\frac{e^\gamma}{\pi}$. Our lower bound relies on the fact that there is only one imprimitive character modulo a prime, but the same lower bound will hold for any $q$ with only a few imprimitive characters; it would be enough to have  the number of imprimitive characters modulo $q$ as small as $o(\phi(q)/(\log q)^A)$ for every $A$.
\end{remark}

\begin{corollary}\label{corollary2b}
Let $g_q(\alpha)$ denote the proportion of nontrivial characters modulo $q$ such that
$\abs{S_\chi(q/2)} \ge \frac{e^\gamma}{\pi} \alpha \sqrt q$, and let
\[
g(\alpha) = \limsup_{ {\substack{ q \rightarrow \infty \\ q \ \mathrm{prime}} }} g_q(\alpha).
\]
Then as
$\alpha \rightarrow \infty$,
\[
    g(\alpha) = \exp\left( -\frac{2e^{\alpha - A - 1}}{\alpha}\left(1 + O(\alpha^{-1/2})\right)\right)
\]
\end{corollary}
\begin{proof}[Proof of Corollary \ref{corollary2b}]
This follows from Theorem \ref{theorem2b} and the argument of
\cite[Proof of Theorem 1]{granville-sound-extreme-zeta}.
\end{proof}
\begin{proof}[Proof of Theorem \ref{theorem2b}]
Recall the Fourier expansion
\[
    \sum_{n \le t} \chi(n) = \frac{\tau(\chi)}{2 \pi i}
        \sum_{1 \le \abs{n} \le q} \frac{\overline{\chi}(n)}{n} \bigg(1 - e\Big(-\frac{nt}{q}\Big)\bigg)
            + O(\log q).
\]
Taking $t = q/2$ and applying Minkowski's inequality, it follows that
\begin{equation}\label{misc-eq-2}
\begin{split}
    \sum_{\chi \ne \chi_0} \abs{S_\chi(q/2)}^{2k}
         &= \Bigg\{\Bigg(\sum_{\chi \ne \chi_0} \abs{\frac{\tau(\chi)}{2 \pi i} \sum_{1 \le \abs{n} \le q} \frac{\bar \chi(n)}{n} \big(1 - e(-n/2)\big)}^{2k}\Bigg)^{\frac{1}{2k}} \phantom{MMMMM} \\
         & \phantom{MMMMMMMMMMMMMMMMMM}
         + O\big(\phi(q)^{\frac{1}{2k}} \log q\big)\Bigg\}^{2k} \\
         &= \left[\left(\sum_{\chi \textrm{ odd}}\frac{2^{2k}q^k}{\pi^{2k}}
         \abs{\sum_{\substack{n \leq q \\ n \textrm{ odd}}} \frac{\chi(n)}{n}}^{2k}\right)^\frac{1}{2k}
                + O\big(\phi(q)^\frac{1}{2k}\log q\big)\right]^{2k}.
\end{split}
\end{equation}
%Taking $N = q$, it follows from this that
%\[
%  \abs{\sum_{1 \le n \le q/2} \chi(n)}^{2k} \le \frac{(4q)^k}{\pi^{2k}} \abs{\sum_{1 \le n \le q, n \textrm{ odd}} \frac{\chi(n)}{n}}^{2k} + 2^{2k}(3 + \log q)^{2k}.
%\]
%and
%\[
%  \abs{\sum_{1 \le n \le q/2} \chi(n)}^{2k} \ge \frac{q^k}{\pi^{2k}} \abs{\sum_{1 \le n \le q, n \textrm{ odd}} \frac{\chi(n)}{n}}^{2k}.
%\]
Using orthogonality of odd characters, we find that
\begin{equation}\label{misc-eq-1}
    \sum_{\chi \textrm{ odd}}\abs{\sum_{\substack{1 \le n \le q \\ n \textrm{ odd}}} \frac{\chi(n)}{n}}^{2k} =
            \frac{\phi(q)}{2} \sum_{a \leq q} \abs{ \sum_{m \geq 0} b(a + mq)}^2
            - \frac{\phi(q)}{2} \sum_{a \leq q} \sum_{s \geq 0} \sum_{t \geq 0} b(a + sq)b(q - a + tq),
\end{equation}
where $b(n) = 0$ if $n$ is even and
\[
    b(n) = \frac{\#\left\{n_1 n_2 \cdots n_k = n, n_j \le q\right\}}{n}
\]
otherwise.
We separate the terms from the first sum on the right hand side of \eqref{misc-eq-1} where $a + mq > q$ and find that they contribute at
most $\phi(q) q^{-1 + \epsilon}$, and as $(a + sq)(q - a + tq) \ge q^2/4$, the same is true of the second sum. Thus, we have
\[
    \sum_{\chi \odd}\abs{\sum_{\substack{1 \le n \le q \\ n \textrm{ odd}}} \frac{\chi(n)}{n}}^{2k} = \frac{\phi(q)}{2}\sum_{n \leq q} b(n)^2 + O(\phi(q) q^{-1 + \epsilon}).
\]
For odd $n \le q$ we have $b(n) = \frac{d_k(n)}{n}$, so as $q \rightarrow \infty$ the sum $\sum_{n \leq q} b(n)^2$ tends to
\[
\sum_{n \geq 0} \frac{d_k(2n + 1)^2}{(2n + 1)^2} = \left(\sum_{a = 0}^\infty \frac{d_k(2^a)^2}{2^{2a}}\right)^{-1}
    \sum_{n \ge 1} \frac{d_k(n)^2}{n^2}.
\]
By \cite[Lemma 4]{granville-sound-extreme-zeta},
\[
    \left(\sum_{a = 0}^\infty \frac{d_k(2^a)^2}{2^{2a}}\right)^{-1} = 4^{-k} \exp\big(O(\log k)\big),
\]
and now estimates from \cite[Proof of Theorem 3]{granville-sound-extreme-zeta} and
Mertens' theorem give
\[
\begin{split}
    \sum_{n \geq 0} \frac{d_k(2n + 1)^2}{(2n + 1)^2}
            &= \frac{1}{4^k} \prod_{p \le k}\left(1 - \frac{1}{p}\right)^{-2k} \exp\left(\frac{2k}{\log k}\left(A + O\left(\frac{1}{\log k}\right)\right)\right) \\
            &= \frac{e^{2k\gamma}}{4^k}\exp\left(2k \log\log k + \frac{2k}{\log k}\left(A + O\left(\frac{1}{\log k}\right)\right) \right).
\end{split}
\]
Putting this into equation \eqref{misc-eq-2} concludes the proof.
\end{proof}

\section{A divisor sum estimate}

In our estimates (see the proof of Lemma \ref{Lem:MomentsOfCharSum}), we will require an upper bound for the divisor sum
\[
    \sum_{n=1}^\infty \frac{d_k(n)^2}{n^{2\sigma}}.
\]
When $\sigma = 1$, Granville and Soundararajan deal with this sum in \cite{granville-sound-extreme-zeta}, and with a
slight modification of their method we can easily give an upper bound for more general $\sigma$.

As $d_k(n)$ is a multiplicative function, we first deal with the case of this sum for powers of a single prime.
It can be written, up to a relatively small error, as a value of the modified Bessel function $I_0(x)$.
\begin{lemma}\label{divisor-sum-bound}
\begin{equation}\label{eq-divisor-bessel}
\log \left(\sum_{a = 0}^\infty \frac{d_k(p^a)^2}{p^{2a\sigma}}\right) =
\log I_0\left(\frac{2k}{p^\sigma}\right) + O\Big(\frac{k}{p^{2\sigma}}\Big)
\end{equation}
where $I_0$ is the modified Bessel function defined by (\ref{eq:Bessel}).
\end{lemma}
\begin{proof}
Note that for any prime $p$, we have the identity
\[
    d_k(p^a) = {k + a - 1 \choose a}.
\]
By applying the binomial theorem to each of
$(1 - e(\theta)/p^\sigma)^{-k}$ and
$(1 - e(-\theta)/p^\sigma)^{-k}$
and multiplying the resulting series, we see that
\[
    \sum_{a = 0}^\infty \frac{d_k(p^a)^2}{p^{2a\sigma}} = \int_0^1 \abs{1 - \frac{e(\theta)}{p^\sigma}}^{-2k} \ud \theta.
\]
It remains only to estimate the integral. To this end, observe that
\[
\begin{split}
    \log \abs{1 - \frac{e(\theta)}{p^{\sigma}}} & = \Re \left\{\log \left(1 - \frac{e(\theta)}{p^{\sigma}}\right)\right\} \\
        & = \Re \left\{  -\frac{e(\theta)}{p^\sigma} - \frac{e(2 \theta)}{2p^{2\sigma}} -  \frac{e(3 \theta)}{3p^{3\sigma}} - \cdots \right\} \\
        & = -\sum_{j=1}^\infty \frac{\cos(2j\pi \theta)}{jp^{j\sigma}} \\
        &= -\frac{\cos(2 \pi \theta)}{p^\sigma} + O(p^{-2\sigma}) .
\end{split}
\]
Hence,
\[
\begin{split}
    \log \int_{0}^1 \abs{1 - \frac{e(\theta)}{p^\sigma}}^{-2k} \ud \theta
        & = \log \int_{0}^1 \exp\left(\frac{2k}{p^\sigma} \cos(2 \pi \theta)\right)\ud \theta + O\Big(\frac{k}{p^{2\sigma}}\Big) \\
        & = \log I_0\left(\frac{2k}{p^\sigma}\right) +  O\Big(\frac{k}{p^{2\sigma}}\Big).
\end{split}
\]
\end{proof}

Now we can give our upper bound. The following is not as precise as it could be, but is sufficient for our purposes,
where $\sigma$ will be very close to $1$.

\begin{proposition}
For $1/2 < \sigma \le 1$,
\[
 \sum_{n=1}^\infty \frac{d_k(n)^2}{n^{2\sigma}} \le \exp\left(2k \sigma \log\log (2k)^{1/\sigma} + \frac{ (2k)^{1/\sigma}}{2\sigma - 1} + O\left( \frac{k}{2\sigma - 1} + \frac{(2k)^{1/\sigma}}{\log(3 (2k)^{1/\sigma - 1} )}\right) \right).
\]
\end{proposition}

\begin{proof}
Since $d_k(n)$ is multiplicative, we have
\[
\begin{split}
    \sum_{n=1}^\infty \frac{d_k(n)^2}{n^{2\sigma}} &= \prod_p \sum_{a = 0}^\infty \frac{d_k(p^a)^2}{p^{2\sigma a} }.
\end{split}
\]
We now apply the previous lemma, and use the fact that $I_0(x) < \min(\exp(x), \exp(x^2))$, to get
\[
\log \sum_{n=1}^\infty \frac{d_k(n)^2}{n^{2\sigma}}  \le 2k\sum_{p \le (2k)^{1/\sigma}} \frac{1}{p^\sigma} + 4k^2 \sum_{p > (2k)^{1/\sigma}} \frac{1}{p^{2\sigma}} + 2k \sum_p \frac{1}{p^{2\sigma}}.
\]
The first sum on the right hand side will be dealt with in Lemma \ref{lemma-prime-sum}, and it has size
\[
    \sigma \log\log (2k)^{1/\sigma} +  O\left( \frac{(2k)^{1/\sigma - 1}}{\log(3 (2k)^{1/\sigma - 1} )}\right).
\]
For the other two sums, we ignore the fact that they are restricted to primes and sum over all integers, obtaining
\[
\log \sum_{n=1}^\infty \frac{d_k(n)^2}{n^{2\sigma}} \le 2k \sigma \log\log (2k)^{1/\sigma} + \frac{ (2k)^{1/\sigma}}{2\sigma - 1} + O\left(\frac{k}{2\sigma - 1} + \frac{(2k)^{1/\sigma}}{\log(3 (2k)^{1/\sigma - 1} )}\right).
\]
\end{proof}

The lemma promised in the above proof is the following.

\begin{lemma} \label{lemma-prime-sum} Uniformly for $1/2 \le \sigma \le 1$,
\[
    \sum_{p < x} \frac{1}{p^\sigma} = \sigma \log\log x + O\left(\frac{x^{1-\sigma}}{\log 3x^{1-\sigma}}\right).
\]
\end{lemma}
\begin{remark}
The proof of this lemma is fairly straightforward. It is adapted from the proof of K. Norton's Lemma 3.1 in \cite{norton}.
\end{remark}
\begin{proof}
Integrating by parts, we have
\[
    \sum_{p < x} \frac{1}{p^\sigma} = x^{-\sigma} \pi(x) + \sigma \int_2^x u^{-(\sigma + 1)} \pi(u) \ud u.
\]
With think of $\sigma$ as being close to one and we approximate the integrand by
\[
    \frac{1}{u \log u},
\]
which integrates to $\log \log x$. So we write
\[
    \sum_{p < x} \frac{1}{p^\sigma} = \sigma \log \log x + x^{-\sigma} \pi(x) +
        \sigma \int_2^x u^{-(\sigma + 1)} \pi(u) - \frac{1}{u \log u} \ud u.
\]
The term $x^{-\sigma} \pi(x)$ is clearly of the size of the error term that we are looking for, so we are left to analyze the integral. It is of size
\[
        \int_2^x \frac{1}{u^\sigma \log u}  - \frac{1}{u \log u} \ud u + O\left(\int_2^x \frac{1}{u^\sigma (\log u)^2}\ud u\right),
\]
where we again think of $\sigma$ as close to $1$ and, and rewrite this as
\[
        \int_2^x \frac{u^{1-\sigma} - 1}{u\log u} \ud u + \frac{1}{\log x} +
            O\left(\int_2^x \frac{u^{1-\sigma} - 1}{u (\log u)^2} \ud u\right).
\]
To finish, we can expand $u^{1-\sigma} - 1$ as a power series and integrate term by term to find that both integrals are
\[
    \ll \frac{x^{1 - \sigma} - 1}{(1 - \sigma)\log x} \ll \frac{u^{1 - \sigma}}{(1-\sigma)\log 3x}.
\]
\end{proof}

\section{An upper bound for the moments of $M(\chi)$}
To deal with the sum
\[
    \sum_{\chi \ne \chi_0} M(\chi)^{2k}
\]
we will use a variant of Theorem \ref{theorem2b} for sums of general length. The method of proof is very similar to that of Theorem \ref{theorem2b}, but there is an additional complication: we have a much worse understanding of the exponentials appearing in the Fourier expansion. Because of this, the method produces only an upper bound.

To deal with the fact that we do not know for which $n$ the maximum of $S_\chi(n)$ is attained, we follow Montgomery and Vaughan in employing the Rademacher--Menchov device. (See \cite{menchov-rademacher-device} for another application.) This approach, which will be described in the proof of Theorem \ref{main-theorem} below, involves dividing the sum $S_\chi(n)$ into dyadic blocks, applying H\"older's inequality, and then summing over all possible dyadic subdivisions.

We have the following upper bound for sums of general length and position.
\begin{lemma}
\label{Lem:MomentsOfCharSum}
There exist an absolute constant $C_1$, and a function $C_2(k, \epsilon)$ independent of $q$, such that
\begin{equation}\label{eq:lemma41}
    \sum_{\chi \ne \chi_0} \abs{ \sum_{\alpha q < n \le \beta q} \chi(n)}^{2k} \le C_1^k \phi(q) q^k \left(\beta - \alpha\right)^{2k/\log k} (\log 2k)^{2k} + C_2(k, \epsilon)\phi(q)q^{k - 1 + \epsilon}
\end{equation}
whenever $0 \le \alpha \le \beta \le 1$.
\end{lemma}
\begin{remark}
We have tailored this bound rather specifically for our purposes. The division by $\log k$ in the exponent of $\beta - \alpha$ comes from an application of Rankin's trick, and is suitable for our application, but there is some flexibility in the method of proof.
\end{remark}

\begin{proof}
We first note that it suffices to prove that the same bound holds for the sum restricted to primitive characters:
\begin{equation}\label{eq:lemma41_primitive}
    \sumstar{\chi \ne \chi_0} \abs{ \sum_{\alpha q < n \le \beta q} \chi(n)}^{2k} \le C_1^k \phi(q) q^k \left(\beta - \alpha\right)^{2k/\log k} (\log 2k)^{2k} + C_2(k, \epsilon)\phi(q)q^{k - 1 + \epsilon},
\end{equation}
where $\sum^*$ denotes that the sum runs over primitive characters. To see this, let $\chi^*$ denote
the primitive character inducing $\chi$ and observe that
\[
\begin{split}
    \sum_{\chi \ne \chi_0} \abs{ \sum_{\alpha q < n \le \beta q} \chi(n)}^{2k} &=
        \sum_{\substack{r | q \\ r > 1}} \sumstar{{\chi^*\!\!\mod r}} \abs{ \sum_{\alpha q < n \le \beta q} \chi^*(n) -
            \sum_{\substack{e | \frac{q}{r} \\ e > 1}} \sum_{ \alpha q/e < m \le \beta q/e } \chi^*(me) }^{2k} \\
        &\le  \sum_{\substack{r | q \\ r > 1}} d(q/r)^{2k} \sum_{e | \frac{q}{r}} \sumstar{\chi^*\!\!\mod r}
            \abs{\sum_{ \alpha q/e < m \le \beta q/e } \chi^*(me) }^{2k}.
\end{split}
\]
Inserting \eqref{eq:lemma41_primitive}, we find that this is bounded by
\[
C_1^k (\beta - \alpha)^{2k/\log k} (\log 2k)^{2k}\sum_{r | q} d(q/r)^{2k + 1} \phi(r) r^k  +
    C_2(k, \epsilon)\sum_{r | q} d(q/r)^{2k + 1} \phi(r) r^{k - 1 + \epsilon}.
\]
Noting that
\[
    \sum_{r | q} d(q/r)^{2k + 1} \phi(r) r^k \le \phi(q) q^k \sum_{n = 1}^\infty \frac{d(n)^{2k + 1}}{n^k},
\]
and $d(n) \ll n^{1/8}$, this is
\[
    \le \phi(q)q^k C^k \sum_{n=1}^\infty n^{-3k/4 + 1/8}
\]
for some absolute constant $C$, and this sum is uniformly bounded for all $k \ge 2$.

To prove \eqref{eq:lemma41_primitive}, we begin with the Fourier expansion
\[
    \sum_{\alpha q < n \le \beta q} \chi(n) = \frac{\tau(\chi)}{2 \pi i} \sum_{1 \leq \abs{n} \le q} \frac{\overline{\chi}(n)}{n} e(\alpha n)\Big(1 - e\big((\beta - \alpha)n\big)\Big)
            + O(\log q) .
\]
It follows from this that
\[
  \abs{\sum_{\alpha q < n \le \beta q} \chi(n)}^{2k} \le \frac{q^k}{\pi^{2k}} \abs{\sum_{1 \le n \le q} \frac{\chi(n)}{n}e(\alpha n)\Big(1 - e\big((\beta - \alpha)n\big)\Big)}^{2k} + 2^{2k}(3 + \log q)^{2k}.
\]
Now define
\[
    b(n) = \sum_{\substack{n_1 n_2 \cdots n_k = n \\ n_i \leq q}} \prod_{j = 1}^k \left( \frac{e(\alpha n_j)}{n_j}  \Big(1 - e\big( (\beta - \alpha)n_j\big)\Big) \right),
\]
for $(n, q) = 1$, and $b(n) = 0$ otherwise,
so that
\[
    \left(\sum_{1 \le n \le q} \frac{\chi(n)}{n}e(\alpha n)\Big(1 - e\big( (\beta - \alpha)n\big)\Big)\right)^k = \sum_{n = 1}^\infty \chi(n) b(n).
\]
Note that
\[
    b(n) \le 2^k d_k(n) \min\left\{ \frac{1}{n}, \pi^k(\beta - \alpha)^k \right\}.
\]
Inserting the nonprimitive characters and using the orthogonality relations for characters, we have
\[
\begin{split}
	\sumstar{\chi} \abs{\sum_{1 \le n \le q} \frac{\chi(n)}{n}e(\alpha n)\Big(1 - e\big( (\beta - \alpha)n\big)\Big)}^{2k} &\le
    \sum_{\chi} \abs{\sum_{1 \le n \le q} \frac{\chi(n)}{n}e(\alpha n)\Big(1 - e\big( (\beta - \alpha)n\big)\Big)}^{2k} \\ &= \phi(q) \sum_{a \leq q} \abs{ \sum_{m = 0}^\infty b(a + mq)}^2.
\end{split}
\]

In the above sum, we separate the parts where $m = 0$ and $m > 0$, to write
\[
    \sum_{a \leq q} \abs{ \sum_{m = 0}^\infty b(a + mq)}^2 \le
    4 \sum_{a \leq q} b(a)^2 + 4 \sum_{a \leq q} \abs{ \sum_{m = 1}^\infty b(a + mq)}^2.
\]
Now note that
\[
\begin{split}
\sum_{a \leq q} \abs{ \sum_{m = 1}^\infty b(a + mq)}^2
    &\ll \sum_{a \leq q} \abs{\sum_{m \leq q^k} \frac{d_k(a + mq)}{a + mq}}^2 \\
    &\ll \frac{1}{q^2} \sum_{a \leq q} \abs{ \sum_{m \leq q^k} \frac{q^\epsilon}{a/q + m}}^2 \\
    &\ll \frac{q^\epsilon}{q} \big(\log (q^k)\big)^2 \\
    &\ll \frac{q^\epsilon}{q}.
\end{split}
\]
Thus we currently have that
\[
    \sumstar{\chi \ne \chi_0} \abs{ \sum_{\alpha q < n \le \beta q} \chi(n)}^{2k} \le \frac{4 \phi(q) q^k}{\pi^{2k}}\sum_{a \leq q} b(a)^2 + O(\phi(q) q^{k - 1 + \epsilon}),
\]
and it remains to analyze one more sum.
We have
\[
    \sum_{a \leq q} b(a)^2 \le \sum_{a \leq q} d_k(a)^2 \min\left\{ \frac{1}{a^2}, \pi^{2k}(\beta - \alpha)^{2k}\right\}.
\]
We split this into $a > \pi^k (\beta - \alpha)^{-k}$ and $a < \pi^k (\beta - \alpha)^{-k}$ to get that this is less than
\[
    \pi^{-2k}(\beta - \alpha)^{2k} \sum_{a \le \pi^k (\beta - \alpha)^{-k}} d_k(a)^2 + \sum_{a > \pi^k (\beta - \alpha)^{-k}}\frac{d_k(a)^2}{a^2}.
\]
Now we use the inequality
\[
X^{-2} \sum_{n \le X} a_n^2 + \sum_{n > X} \frac{a_n^2}{n^2} \le X^{2\sigma - 2} \sum_{n=1}^\infty \frac{a_n^2}{n^{2\sigma}},
\]
and we have
\[
    \pi^{-2k}(\beta - \alpha)^{2k} \sum_{a \le \pi^k (\beta - \alpha)^{-k}} d_k(a)^2 + \sum_{a > \pi^k (\beta - \alpha)^{-k}}\frac{d_k(a)^2}{a^2} \le
        \left(\frac{\pi^{k}}{(\beta - \alpha)^k}\right)^{2\sigma - 2} \sum_{n=1}^\infty \frac{d_k(n)^2}{n^{2\sigma}}.
\]
Inserting the bound from Lemma \ref{divisor-sum-bound}, this is
\[
 \le  \left(\frac{\pi^{k}}{(\beta - \alpha)^k}\right)^{2\sigma - 2} \exp\left(2k \sigma \log\log (2k)^{1/\sigma} + \frac{ (2k)^{1/\sigma}}{2\sigma - 1} + O\left( \frac{k}{2\sigma -1} + \frac{(2k)^{1/\sigma}}{\log(3 (2k)^{1/\sigma - 1} )}\right) \right).
\]
Choosing $\sigma = 1 - 1/\log k$ gives
\[
    \sumstar{\chi \ne \chi_0} \abs{ \sum_{\alpha q < n \le \beta q} \chi(n)}^{2k} \le \frac{4 \phi(q) q^k}{\pi^{2k}} \left(\frac{\pi}{\beta - \alpha}\right)^{-\frac{2k}{\log k}}\!\!\exp\!\big(2k \log\log 2k + O(k)\big) + O(\phi(q)q^{k - 1 + \epsilon}).
\]
\end{proof}

With this result in hand, we proceed to the proof of an upper bound for the moments of $M(\chi)$. We recall the statement of our main theorem.

\begin{maintheorem}
For fixed $k$ and $q \rightarrow \infty$
\[
\frac{1}{\phi(q)} \sum_{\chi \ne \chi_0} M(\chi)^{2k} \le \big(C(k) + o(1)\big) q^k,
\]
where
\[
    C(k) = \exp\big(4k \log\log k + k \log\log\log k + O(k)\big).
\]
\end{maintheorem}

\begin{proof}
To prove this theorem we will partition the range of summation for $M(\chi)$ and then apply the above lemma.
However, since $M(\chi)$ will occur in different spots for different $\chi$, we do not know exactly how to break up the range of summation, so, as mentioned previously, we use the method of Menchov and Rademacher
to split the sum in many different ways. It turns out that we can split the summation in sufficiently many ways to capture all of the possibilities for $M(\chi)$ without splitting it up in so many different ways that we lose our theorem.

To begin with, for a given $\chi$, let $N_\chi$ be the smallest integer such that
\[
    M(\chi) = \abs{\sum_{n \le N\chi} \chi(n)}.
\]
Write $N_\chi/ q$ in base $2$, so that
\[
     \frac{N_\chi}{q} = \sum_{j=1}^\infty \frac{a_j}{2^j},
\]
where $a_j = 0$ or $1$. Let $A_\chi(n)$ be a truncation of this series, which will serve as an approximation to $N_\chi/q$:
\[
    A_\chi(n) = \sum_{j \leq n} \frac{a_j}{2^j}.
\]
Then we have
\[
\begin{split}
    M(\chi) &= \abs{\sum_{n \le q A_\chi(L)} \chi(n)} + O(q/{2^L}) \\
            &= \abs{\sum_{l = 1}^{L - 1} \ \sum_{q A_\chi(l) < n \le q A_\chi(l + 1)} \chi(n)} + O(q/2^L).
\end{split}
\]

Using H\"older's inequality, we thus have that
\[
    M(\chi)^{2k} \le 2^{2k} \left(\sum_{l \leq L} \frac{1}{l^{\alpha + \frac{\alpha}{2k - 1}}}\right)^{2k - 1} \left( \sum_{l \leq L} l^{2k\alpha} \abs{\sum_{qA_\chi(l) < n \le qA_\chi(l + 1)} \chi(n) }^{2k}\right) + O(q^{2k}/2^{2kL}).
\]
In order to take the sum over characters and make use of orthogonality, we need the intervals of summation
to be independent of $\chi$. Note that either $A_\chi(l+1) = A_\chi(l)$ (in which case the innermost sum is empty) or $A_\chi(l+1) = A_\chi(l) + 2^{-(l+1)}$. We can therefore write
\[
    M(\chi)^{2k} \le 2^{2k} \left(\sum_{l \leq L} \frac{1}{l^{\alpha + \frac{\alpha}{2k - 1}}}\right)^{2k - 1} \left( \sum_{l \leq L} l^{2k\alpha} \abs{\sum_{qA_\chi(l) < n \le q(A_\chi(l) + 2^{-(l + 1)})} \chi(n) }^{2k}\right) + O(q^{2k}/2^{2kL}).
\]
Now, we don't know what the value of $A_\chi(l)$ is, but there are ``only'' $2^{l-1}$ possibilities, so we just include them all to get
\begin{multline*}
    M(\chi)^{2k} \le 2^{2k} \left(\sum_{l \leq L} \frac{1}{l^{\alpha + \frac{\alpha}{2k - 1}}}\right)^{2k - 1} \left( \sum_{l \leq L} l^{2k\alpha} \sum_{0 \leq m \leq 2^l-1}
    \abs{\sum_{q m/2^l < n \le q(m/2^l + 2^{-(l + 1)})} \chi(n) }^{2k}\right)
	\\
	\\ + O\big(q^{2k}/2^{2kL}\big).
\end{multline*}
Finally we are in a position to apply Lemma \ref{Lem:MomentsOfCharSum}, obtaining
\begin{multline*}
    \sum_{\chi \ne \chi_0} M(\chi)^{2k} \le
    2^{2k} C_1^k\phi(q)q^k
    \left(\sum_{l \leq L} \frac{1}{l^{\alpha + \frac{\alpha}{2k - 1}}}\right)^{2k - 1}
                \sum_{l \leq L} l^{2k\alpha} 2^l \left(2^{-\frac{kl}{\log k}} (\log 2k)^{2k} +
	\frac{C_2(k, \epsilon)}{q^{1 - \epsilon}}\right) \\ \\ + O\big(\phi(q)q^{2k}/2^{2kL}\big).
\end{multline*}
Rearranging a bit, inserting $L = \floor{\log_2 q^{3/4}}$, and extending the convergent sums to infinity, we rewrite this as
\begin{multline*}
    \sum_{\chi \ne \chi_0} M(\chi)^{2k} \le
    \phi(q) q^k \exp\!\big(2k \log\log 2k + O(k)\big) H(\alpha) \sum_{l=1}^\infty l^{2k \alpha} 2^{l} 2^{-\frac{kl}{\log k}} \\
	+ O_{k, \epsilon}\Big(\phi(q) q^{k - 1/4 + \epsilon}\Big) + O\Big(\phi(q)q^{k/2}\Big),
\end{multline*}
where $H(\alpha) = \left(\sum_{l=1}^\infty \frac{1}{l^{\alpha + \frac{\alpha}{2k - 1}}}\right)^{2k - 1}$.

In the remaining sum we note that for $l > C\log k \log\log k$ each summand
is decreasing as $k$ increases. The maximum summand occurs near
\[
    l = \frac{2 \alpha \log k}{\left(1 + \frac{\log k}{k}\right)\log 2},
\]
and at this point it has size
\[
    \left(\frac{2\alpha \log k}{ (1 + (\log k)/k) \log 2}\right)^{2k\alpha} k^{-2\alpha/(1 + (\log k)/k)} 2^{-k/\log k}.
\]
Hence, we find that
\[
\sum_{l \leq L} l^{2k\alpha}2^{l - \frac{kl}{\log k}} \le \exp\big(2k\alpha \log \log k + O(k)\big).
\]
Finally, $H(\alpha) \sim (\alpha - 1)^{1 - 2k}$, so taking $\alpha = 1 + 1/\log\log k$ yields the theorem.
\end{proof}
\begin{proof}[Proof of Theorem \ref{main-corollary}]
We now use our upper bound on the moments of $M(\chi)$ to obtain an upper bound on the number of
characters that can get large. As before, let $F_q(\alpha)$ denote the fraction of characters $\chi \mod q$
for which $M(\chi) \le \alpha \sqrt{q}$. Then we have
\[
    \alpha^{2k} \big(1 - F_q(\alpha)\big) \phi(q)q^k \le \sum_{\chi \ne \chi_0} M(\chi)^{2k} .
\]
So as $q \rightarrow \infty$, from Theorem \ref{main-theorem} we have
\[
	\alpha^{2k} \big(1 - F_q(\alpha)\big) \phi(q)q^k \le \big(C(k) + o(1)\big)\phi(q)q^k.
\]
Hence for some large enough constant $C$,
\[
    \alpha^{2k} (1 - F(\alpha)) \le \exp\big(4k\log\log k + k \log\log\log k + Ck\big).
\]
To finish, we just need to choose $k$ depending on $\alpha$ to get a good
bound. Making the choice
\[
    k = \ceil{\exp\left(\frac{B\alpha^{1/2}}{(\log \alpha)^{1/4}}\right)},
\]
where $B = 2^{1/4}e^{-C/4}$, we have
\[
    \log k = \frac{B \alpha^{1/2}}{(\log \alpha)^{1/4}} + O\bigg(\frac{1}{k}\bigg),
\]
\[
    \log \log k = \frac{1}{2} \log \alpha - \frac{1}{4}\log \log \alpha + \log B + O\bigg(\frac{1}{k}\bigg)
\]
and
\[
    \log \log \log k = -\log 2 + \log\log \alpha + \log\left(1 - \frac{2 \log \log \alpha}{\log \alpha}
                        + \frac{\log B}{2 \log \alpha} \right) + O\bigg(\frac{1}{k}\bigg).
\]
Putting these in gives
\begin{multline*}
    -2k \log \alpha + 4k \log\log k + k\log\log\log k + Ck \le \\
        -\exp\bigg( \frac{B \alpha^{1/2}}{(\log \alpha)^{1/4}}\bigg)\ \Bigg(\frac{2\log\log \alpha}{\log \alpha} + O\bigg(\frac{1}{\log \alpha}\bigg) \Bigg).
\end{multline*}
Hence
\[
F(\alpha) \ge 1 - 
    \exp\Bigg(-\exp\bigg(\frac{2^{1/4}e^{-C/4}\alpha^{1/2}}{(\log \alpha)^{1/4}}\bigg)
        \ \bigg(\frac{2\log\log\alpha}{\log\alpha}
        + O\Big(\frac{1}{\log \alpha}\Big)\bigg)\Bigg).
\]
This gives the claimed lower bound.

The upper bound on $F(\alpha)$ comes directly from the lower bound of Corollary \ref{corollary2b} on the number of characters for which $\abs{S_\chi(q/2)}$ is large.

\end{proof}

\section{Concluding remarks}

Our bounds on the frequency with which $M(\chi)$ is large should be compared to results of Granville
and Soundararajan on the distribution of $L(1, \chi)$. They prove
\begin{Theorem}[Granville--Soundararajan \protect{\cite[Theorem 3]{granville-sound-extreme-zeta}}]
Let $q$ be a large prime.
The proportion of characters $\chi$ mod $q$ such that $\abs{L(1, \chi)} > e^\gamma \tau$ is
\[
    \exp\left(-\frac{2 e^{\tau - A - 1}}{\tau}\left(1 + O\left(\frac{1}{\tau^{1/2}} + \left(\frac{e^\tau}{\log q}\right)^{1/2}\right)\right)\right)
\]
uniformly in the range $1 \ll \tau \le \log \log q - 20$.
\end{Theorem}

For comparison, we might consider results that relate the value of $L(1, \chi)$ to the size of the
character sum. When $\chi(-1) = 1$, $S_\chi(q/2)$ vanishes. Otherwise, we have the relation
\[
    S_\chi(q/2) = \big(2 - \chi(2)\big)\frac{\tau(\chi)}{\pi i} L(1, \overline \chi).
\]
Thus $M(\chi)$ gets at least as large as a constant times $q^{1/2}L(1, \overline \chi)$ (at least
for $\chi$ odd), so we might expect that the proportion of $\chi$ for which
$M(\chi) > \alpha \sqrt q$ should be ``like''  $\exp(-\exp \alpha)$. Indeed, this is precisely where
our lower bound on the number of characters that get large comes from, and our upper bound has a
similar shape.
One could also obtain such an upper bound for
$S_\chi(tq)$ for any fixed
$t \in [0,1]$. On the other hand, the maximum of the character sum does
not always occur at the central point, and it is precisely in considering all of the places where
$M(\chi)$ may occur that we lose in the exponent; this is in our application of H\"older's inequality.
It is not completely clear what the true size of the moments of $M(\chi)$ should be.

\end{document}